\documentclass[12pt,letterpaper]{article}
\usepackage{authblk}
\usepackage{hyperref}
\usepackage{geometry}
\usepackage[latin1]{inputenc}
\usepackage[italian,english]{babel}
\usepackage{amsmath, amsfonts, amsthm}
\usepackage{paralist}
\usepackage{bbm}
\numberwithin{equation}{section}
\newtheorem{thm}{Theorem}[section]
\newtheorem{lem}[thm]{Lemma}
\newtheorem{cor}[thm]{Corollary}
\newtheorem{prop}[thm]{Proposition}

\newtheorem{defn}[thm]{Definition}
\theoremstyle{definition}

\theoremstyle{remark}

\newcommand{\ds}{\displaystyle}

\newcommand{\de}{\partial}

{\left\{\begin{array}{@{}l@{}}}{\end{array}\right.}
\patchcmd{\abstract}{\scshape\abstractname}{\textbf{\abstractname}}{}{}
\makeatletter 
\def\@makefnmark{} 
\makeatother 

\date{}

\title{Some properties of the torsion function with Robin boundary conditions}
\author{Rossano Sannipoli}
\affil{Dipartimento di Matematica e Applicazioni ``R. Caccioppoli'', Universit\`a degli studi di Napoli Federico II \\ Via Cintia, Complesso Universitario Monte S. Angelo, 80126 Napoli, Italy. \\
\href{rossano.sannipoli@unina.it}{rossano.sannipoli@unina.it}}

\begin{document}
Partial Differential Equations
{\let\newpage\relax\maketitle}
\begin{abstract}
	In this paper we study some properties of the torsion function with Robin boundary conditions. Here we write the shape derivative of the $L^{\infty}$ and $L^p$ norms, for $p\ge 1$, of the torsion function, seen as a functional on a bounded simply connected open set $\Omega \subset \mathbb{R}^n$, and prove that the balls are critical shapes for these functionals, when the volume of $\Omega$ is preserved.
	\\
	\textsc{MSC 2010:} 35J05 - 35J15 - 35J20 - 35J25 \\
	\textsc{Keywords:} Torsion problem - Robin boundary conditions - Shape derivative
\end{abstract}

\section{Introduction}
Let $\Omega\subset \mathbb{R}^n$, $n \ge 2$, be a bounded $C^{2,\alpha}$ and simply connected open set. Let us consider the following torsion problem with Robin boundary conditions
\begin{equation} \label{TPR}
\begin{cases}
-\Delta u=1 & \mbox{in}\ \Omega\vspace{0.2cm}\\
\dfrac{\de u}{\de \nu}+\beta u=0&\mbox{on}\ \partial \Omega\vspace{0.2cm},
\end{cases}
\end{equation}
where $\nu$ stands for the outer unit normal to $\partial\Omega$ and $\beta> 0$ is a positive real number, known as Robin boundary parameter.\\
A weak solution to \eqref{TPR} is a function $u\in H^1(\Omega)$ which satisfies
\begin{equation*}
\int_{\Omega} \nabla u \nabla \phi \,dx + \beta\int_{\partial\Omega} u \phi \,d\mathcal{H}^{n-1} = \int_{\Omega} \phi, \,\,\,\,\,\,\,\,\,\,\,\,\,\, \forall \phi \in H^1(\Omega).
\end{equation*}
When $\beta \to \infty$ we have the well known torsion problem of elasticity or Saint Venant problem, that has been studied by many authors: different estimates and qualitative properties have been studied, for example, by P\'{o}lya and Szeg\"{o} in [PS], or Payne in [P].\\
Other problems with non-linear operators and a different geoemtry of the domain can be found. One example is given in [PPT]: if $\Omega$ is a convex set with a hole, the authors prove that the spherical shell is a minimizer for the torsional rigidity of the $p$-Laplacian, with outer Robin and internal Neumann boundary conditions, when the volume and the outer perimeter are fixed.\\
It is well known that the solution to problem \eqref{TPR} is unique and positive whenever $\partial \Omega$ is sufficiently smooth.\\
A comparison result à la Talenti has been proved in [ANT]. They consider the symmetrized problem, that is

\begin{equation} \label{TPRS}
\begin{cases}
-\Delta v=1 & \mbox{in}\ \Omega^{\sharp}\vspace{0.2cm}\\
\dfrac{\de v}{\de \nu}+\beta v=0&\mbox{on}\ \partial \Omega^{\sharp}\vspace{0.2cm},
\end{cases}
\end{equation}
 where $\Omega^{\sharp}$ is the ball centered in the origin having the same measure as $\Omega$. In dimension $2$, they proved that
 \begin{equation*}
 u^{\sharp} (x)\le v(x)  \;\;\;\;\;\;\; x \in \Omega^{\sharp},
 \end{equation*} 
 where $v$ is the solution to the problem \eqref{TPRS} and $u^{\sharp}$ is the Schwarz symmetrization of $u$ (for more details see [K]). Moreover they proved, for $n\ge 3$, that
 \begin{equation} \label{comp1}
 \|u\|_{L^{p,1}(\Omega)}\le \|v\|_{L^{p,1}(\Omega^{\sharp})}
 \end{equation}  
 and
 \begin{equation*} 
 \|u\|_{L^{2p,2}(\Omega)}\le \|v\|_{L^{2p,2}(\Omega^{\sharp})},
 \end{equation*}
 for all $0<p\le \frac{n}{n-2}$. Here $L^{p,q}(\Omega)$ is the Lorentz space (for more detail see [L]). It follows that in dimensions greater than 3
 \begin{equation*}
 	\| u \|_{L^p(\Omega)} \le \| v \|_{L^p(\Omega^{\sharp})}
 \end{equation*} 
 when $p=1,2$.\\
In [BG] the authors proved, with different arguments, that $\|u \|_{L^1(\Omega)} \le \|v\|_{L^1(\Omega^{\sharp})}$, where $u$ is the solution to \eqref{TPR} and $v$ is the solution to the symmetrized problem \eqref{TPRS}.

Problem \eqref{TPR} is a particular case of the following
\begin{equation} \label{TPRG}
	\begin{cases}
		-\Delta u=f & \mbox{in}\ \Omega\vspace{0.2cm}\\
		\dfrac{\de u}{\de \nu}+\beta u=0&\mbox{on}\ \partial \Omega\vspace{0.2cm},
	\end{cases}
\end{equation}

where $f \in L^2(\Omega)$ and its symmetrized problem is \begin{equation} \label{TPRGS}
\begin{cases}
-\Delta v=f^{\sharp} & \mbox{in}\ \Omega^{\sharp}\vspace{0.2cm}\\
\dfrac{\de v}{\de \nu}+\beta v=0&\mbox{on}\ \partial \Omega^{\sharp}\vspace{0.2cm},
\end{cases}
\end{equation} 
 where $f^{\sharp}$ is the Schwarz symmetrization of $f$. When in \eqref{TPRG} and \eqref{TPRGS} we have Dirichlet boundary conditions, Talenti (see [T]), via the rearrangement of a function and the Schwarz symmetrization, showed the following 
 \begin{equation} \label{compf}
 	u^{\sharp}(x) \le v(x) \;\;\;\;\; x\in \Omega^{\sharp}
 \end{equation}
Still in [ANT], a Talenti comparison result for the problem \eqref{TPRG} has been proved: when $n\ge 2$
\begin{equation*}
	\|u\|_{L^{p,1}(\Omega)} \le \|v\|_{L^{p,1}(\Omega^{\sharp})} \,\,\,\,\,\,\, \text{for all} \,\,\, 0<p \le \frac{n}{2n-2},
\end{equation*} 
and
\begin{equation*}
\|u\|_{L^{2p,2}(\Omega)} \le \|v\|_{L^{2p,2}(\Omega^{\sharp})}
\,\,\,\,\,\,\, \text{for all} \,\,\, 0<p \le \frac{n}{3n-4},
\end{equation*}
where $v$ is solution to \eqref{TPRGS}. So, in dimension 2, we have that
\begin{equation} \label{compf}
\| u \|_{L^p(\Omega)} \le \| v \|_{L^p(\Omega^{\sharp})},
\end{equation} 
when $p=1,2$. One may ask if \eqref{compf} is still true for larger values of $p$ in dimension 2 or if it is valid in every dimension and value of $p$. The authors, though, found counterexamples of the untruthfulness of these questions when $n=2$ and $p= \infty$, and when $n=3$ and $p=2$.
This led to the following open problems: 
\begin{itemize}
\item $u^{\sharp}\le v$ in $\Omega^{\sharp}$ for $n\ge 3$ and $f\equiv 1$ ;\\
\item $\|u\|_{L^1(\Omega)}\le \|v\|_{L^1(\Omega^{\sharp})}$ for $n\ge 3$ and $f\in L^2(\Omega)$.
\end{itemize}
 In this paper we move the first steps in these directions.\\

In particular we set
\begin{equation*} 
M(\Omega) = \|u\|_{L^\infty (\Omega)},
\end{equation*}
and for every $p\in [1,+\infty)$ we denote the following functional
\begin{equation*} \label{LPF}
F_p(\Omega) = \int_{\Omega} |u(x)|^p \,dx = \int_{\Omega} u^p(x)\,dx = \|u\|_{L^p(\Omega)}^p,
\end{equation*}
where $u$ is solution to \eqref{TPR}. We are interested in computing the shape derivative (see [HP]) of these two functionals and prove that the ball centered at the origin is a critical shape for them.\\ Namely, if $\Omega\subset \mathbb{R}^n$ is a bounded $C^{2,\alpha}$ simply connected open set, let us consider a first order perturbation
\begin{equation*}
	\Omega_t = (\mathbbm{1}_{\mathbb{R}^n}+tV)(\Omega),
\end{equation*}
with $\mathbbm{1}_{\mathbb{R}^n}$ being the identity function, $V$ a $C^{2,\alpha}(\mathbb{R}^{n},\mathbb{R}^{n})$ vector field and $t$ a small real number. We are interested in the study of the limits (if they exist)
\begin{equation}
	M'(\Omega,V)= \lim_{t\to 0 } \frac{M(\Omega_t)-M(\Omega)}{t}
\end{equation}
and
\begin{equation}\label{SDFF}
F_p'(\Omega,V)= \lim_{t\to 0 } \frac{F_p(\Omega_t)-F_p(\Omega)}{t}.
\end{equation}
Let us denote by $B_R$ the ball centered at the origin in $\mathbb{R}^n$ with radius $R>0$. Then problem \eqref{TPR} becomes 
\begin{equation} \label{SDB}
\begin{cases}
-\Delta u=1 & \mbox{in}\ B_R\vspace{0.2cm}\\
\dfrac{\de u}{\de \nu}+\beta u=0&\mbox{on}\ \partial B_R\vspace{0.2cm},
\end{cases}
\end{equation}
where $\nu = \frac{x}{R}$ is the outer unit normal to the boundary. \\
We are now able to state the main theorem.\\

{\bf Main result.} {\it
	The ball $B_R$ is a critical shape for the functionals $M(\Omega)$ and $F_p(\Omega)$, $p\ge 1$, i.e.}
	\begin{equation*}
	M'(B_R,V)=F'_p(B_R,V)=0,
	\end{equation*}
{\it where $V$ is a $C^{2,\alpha}(\mathbb{R}^n,\mathbb{R}^n)$ vector field volume preserving of the first order and where $M'(\cdot, v)$ and $F_p '(\cdot, v)$ are the shape derivatives of $M$ and $F_p$ respectively.}\\

\vspace{1pt}
For the precise definition of vector field volume preserving of the first order see section 2.

Next section will include some preliminary results in order to compute the shape derivative of the two functionals defined before. In section $3$ there will be the computation of the $L^{\infty}$-norm of the torsion function with the consequent proof of the main result. In section $4$, the same thing for the $L^p$-norm will be done.
\section{Preliminaries}
Throughout this paper we will denote by $B_R= \{x\in \mathbb{R}^n : \|x\| < R\}$ the ball centered at the origin with radius $R>0$, where $\|\cdot \|$ is the classical euclidean distance; by $\Omega$ a bounded $C^{2,\alpha}$ and simply connected open set with finite Lebeasgue measure, where $C^{2,\alpha}$ stands for the $\alpha$-H\"olderian space. We denote by $\mathcal{H}^{n-1}$ the $(n-1)$-dimensional Hausdorff measure in $\mathbb{R}^n$ and by $|\cdot|$ the Lebeasgue measure in $\mathbb{R}^n$. If $x,y \in \mathbb{R}^n$ we will indicate by $x\cdot y$ the scalar product between the two vectors.\\

\subsection{Shape derivative: some definitions and computations}
Here we give some preliminary definitions and results that the reader can find in [BW] and [HP]. We point out that in this and next subsection, we will use the Einstein summation convention for the repeated indexes.\\
Let $\Omega \subset \mathbb{R}^n$ be a bounded and simply connected open set. We consider a family of perturbations $\{\Omega_t \}_t$ of the form
\begin{equation} \label{PERT}
\Omega_t = \ds \{ y=x+tV(x) : x\in \Omega, t \,\, \text{small enough}\},
\end{equation}
where $V$ is a $C^{2,\alpha} (\mathbb{R}^n, \mathbb{R}^n)$ vector field.

The Jacobian matrix of the transformation 
\begin{equation} \label{ytrans}
	y:= y(t,\Omega) = x + tV(x), \,\,\,\,\,\, x\in \Omega, \,\, t \,\, \text{small}
\end{equation} 
is 
\begin{equation*}
D_y=I+tD_V,
\end{equation*}
where $I$ is the identity matrix and $(D_V)_{ij} = \frac{\partial V_i}{\partial x_j}$. By Jacobi's formula, for small t, the Jacobian determinant is given by
\begin{equation} \label{JD}
J(t) = 1 + t\,\mathrm{div}\, V.
\end{equation}
It is clear that for t small enough, $J(t) \approx 1$, so $y(t,\Omega)$ is a diffeomorphism and in this case we can consider its inverse transformation $x(y)$.\\
In particular we can write the measure of $\Omega_t$ in terms of the perturbations defined before
\begin{equation*}
|\Omega_t| = \int_{\Omega} J(t) \, dx = |\Omega| + t \int_{\Omega} \mathrm{div}\,V \,dx.
\end{equation*}

\begin{defn}
$y(t,\Omega)$ is said to be \textit{volume preserving of the first order} if
\begin{equation*}
\int_{\Omega} \mathrm{div}\,V \,dx=0.
\end{equation*}
\end{defn}

Let $w \in H^1(\Omega_t)$ and let us consider the following energy functional
\begin{equation}\label{RE}
\mathcal{E}(\Omega_t, w)= \int_{\Omega_t} |\nabla_y w(y,t)|^2\,dy -2 \int_{\Omega_t} w(y,t)\,dy + \beta\int_{\partial \Omega_t} w^2(y,t) \, d\sigma_t,
\end{equation}
where with $\nabla_y$ we denoted the gradient operator with respect to $y$ and $d\sigma_t$ is the surface element of $\Omega_t$.\\
 A critical point $u\in H^1(\Omega_t)$ of \eqref{RE} satisfies the Euler-Lagrange equations
\begin{equation}\label{ELR}
\begin{cases}
\Delta_y u(y,t) +1 = 0 & \mbox{in}\ \Omega_t\vspace{0.2cm}\\
\dfrac{\de u}{\de \nu_t}(y,t)+\beta u(y,t)=0&\mbox{on}\ \partial \Omega_t\vspace{0.2cm},
\end{cases}
\end{equation}
where $\Delta_y$ is the Laplacian operator with respect to $y$ and $\nu_t$ is the outer normal to $\partial \Omega_t$.\\
We want to transform the integrals in \eqref{RE} in integrals onto $\Omega$ and $\partial \Omega$. Indeed by a change of variables, using the inverse function $x(y)$ 
(which exists for small t), we get
\begin{equation} \label{EFT1}
\begin{split} 
\mathcal{E}(\Omega,u) &= \int_{\Omega} \frac{\partial u}{\partial x_i} (x+tV(x),t) \frac{\partial u}{\partial x_j} (x+tV(x),t)\frac{\partial x_i}{\partial y_k}\frac{\partial x_j}{\partial y_k} J(t)\,dx\\
& -2 \int_{\Omega} u(x+tV(x),t)J(t)\,dx + \beta \int_{\partial \Omega} u^2(x+tV(x),t)m(t) \,d\sigma.
\end{split}.
\end{equation}
Here $m(t)$ is the index of deformation when passing from $d\sigma$ to $d\sigma_t$ (with $d\sigma$ being the surface element of $\Omega$). If we define the tangential divergence of the vector field V as follows
\begin{equation} \label{tangdiv}
\mathrm{div}_{\partial\Omega}\,V := \mathrm{div}\,V - \nu\cdot D_V \nu,
\end{equation}
then, up to first order terms, $m(t)$ can be approximated by (See [BW], section $2.2.2$) 
\begin{equation} \label{indexdef}
m(t) = 1 + t\mathrm{div}_{\partial \Omega}\,V.
\end{equation}

If we denote by
\begin{equation} \label{deftildeu}
\tilde{u}(t):= u(x+tV(x),t)
\end{equation} 
and
\begin{equation} \label{defA}
A=(A_{ij}(t)):= \frac{\partial x_i}{\partial y_k} \frac{\partial x_j}{\partial y_k} J(t),
\end{equation}
we can write \eqref{EFT1} in a more concise form
\begin{equation} \label{EFT}
\mathcal{E}(t) := \int_{\Omega} \nabla \tilde{u}(t) A \nabla \tilde{u}(t)\,dx -2 \int_{\Omega} \tilde{u}(t)J(t)\,dx + \beta \int_{\partial \Omega} \tilde{u}^2(t)m(t) \,d\sigma.
\end{equation}
If we simplify one more time the notations and indicate by
\begin{equation} \label{defLA}
L_A = \frac{\partial }{\partial x_j}(A_{ij}(t)\frac{\partial }{\partial x_i})
\end{equation}
and
\begin{equation} \label{defNUA}
 \de_{\nu_A} = \nu_i A_{ij}(t)\frac{\partial }{\partial x_j},
\end{equation}
then the transformed function $\tilde{u}(t)$ solves the Eulero-Lagrange equations
\begin{equation} \label{ELRT}
\begin{cases}
L_A \tilde{u}(t) +J(t) = 0 & \mbox{in}\ \Omega\vspace{0.2cm}\\
\de_{\nu_A}\tilde{u}(t)+\beta m(t) \tilde{u}(t) =0&\mbox{on}\ \partial \Omega\vspace{0.2cm}.
\end{cases}
\end{equation}
A key role will be played by \eqref{deftildeu}. If we expand it in a neighbourhood of $t=0$ we have
\begin{equation*}
\tilde{u}(t) = \tilde{u}(0)+ t \dot{\tilde{u}}(0) + o(t).
\end{equation*}
We remark that the dot notation stands for the derivative with respect to t. In particular the first two coefficients will be 
\begin{equation} \label{expu1}
\tilde{u}(0):= u(x)
\end{equation}
and
\begin{equation} \label{expu2}
\begin{split}
\dot{\tilde{u}}(0) = \left[ \frac{d}{dt}\tilde{u}(t) \right]_{t=0} &= \left[ \frac{\partial \tilde{u}}{\partial t}(t)+ V(x)\cdot \nabla \tilde{u}(t)\right]_{t=0}\\
&= \frac{\partial \tilde{u}}{\partial t}(0) + V\cdot \nabla u.
\end{split}
\end{equation}
\begin{defn}
	We will call shape derivative of $\tilde{u}$, and it will be denoted by $u'$, the following function
	\begin{equation}
	u'(x) := \frac{\partial \tilde{u}}{\partial t}(0).
	\end{equation}
\end{defn}
Hence we can write \eqref{expu2} in this way
\begin{equation} \label{expu3}
	\dot{\tilde{u}}(0)= u' + V\cdot \nabla u.
\end{equation}
Besides the expansion of $\tilde{u}(t)$, it will be helpful to write the taylor series of some other of the function seen until now. Next Lemma will collect all the necessary coefficients of the expansions just mentioned
\begin{lem} \label{EXPANSIONS}
We have that
\begin{equation}
\begin{aligned}
&1)& &J(0)= 1, & &2)& &\dot{J}(0)=\mathrm{div}\,V, \\
&3)& &m(0)= 1, & &4)& &\dot{m}(0)=\mathrm{div}_{\partial \Omega}\,V, \\
&5)& &\tilde{u}(0)= u(x), & &6)& &\dot{\tilde{u}}(0)= u'+V\cdot \nabla u, \\
&7)& &A_{ij}(0)= \delta_{ij}, & &8)& &\dot{A}_{ij}(0)= \mathrm{div}\, V \delta_{ij} - \frac{\partial V_i}{\partial x_j} -\frac{\partial V_j}{\partial x_i}.
\end{aligned}
\end{equation}

\end{lem} 
\begin{proof}
To compute $1), 2), 3), 4)$ it is sufficient to differentiate \eqref{JD} and \eqref{indexdef} and evaluate for $t=0$. $5)$ and $6)$ are given by \eqref{expu1} and \eqref{expu3} respectively.\\
Some more effort will be needed for the matrix $A$, defined in \eqref{defA}. 
Remembering that the Jacobian matrix of the transformation $y(t,\Omega)$ is
\begin{equation*}
	D_y = I + tD_V,
\end{equation*}
if $t$ is small enough, we have
\begin{equation*}
\begin{split}
	\frac{\partial x_i}{\partial y_k} = (D_y^{-1})_{ik} &= (I+tD_V)^{-1}_{ik} \\
	&= (I-tD_V  + o(t) )_{ik} 
	= \delta_{ik} -t\frac{\partial V_k}{\partial x_i}+  o(t).
\end{split}
\end{equation*}
This allows us to obtain
\begin{equation*}
	A_{ij}(0)= \delta_{ik}\delta_{jk}= \delta_{ij}
\end{equation*}
and
\begin{equation*}
\begin{split}
	\dot{A}_{ij}(0)&= \bigg[\frac{\partial x_i}{\partial y_k} \frac{\partial x_j}{\partial y_k} \dot{J}(t)+ \bigg(\frac{d}{dt}\frac{\partial x_i}{\partial y_k}\bigg)\frac{\partial x_j}{\partial y_k}J(t)+\frac{\partial x_i}{\partial y_k}\bigg(\frac{d}{dt}\frac{\partial x_j}{\partial y_k}\bigg)J(t)\bigg]_{t=0} \\
	&= \mathrm{div}V \delta_{ij} - \frac{\partial V_i}{\partial x_j} -\frac{\partial V_j}{\partial x_i}.
	\end{split}
\end{equation*}
\end{proof}
We want to find the equations that are solved by $u'$ in $B_R$ and on its boundary.

\subsection{An equation for $u'$ in $B_R$}
Let us consider problem \eqref{SDB}. It is well known that it admits a unique and positive solution, given by 
\begin{equation} \label{RTPR}
u(x) = \frac{R}{\beta n} +\frac{1}{2n}(R^2-|x|^2),
\end{equation}
which is a radial and strictly concave function, whose maximum and minimum are achieved in $0$ and on $\partial B_R$ respectevely. More precisely
\begin{equation} \label{umaxumin}
	u_{\max}= u(0)= \frac{R}{\beta n} +\frac{R^2}{2n}, \,\,\,\,\,\,\, u_{\min} = u(R) = \frac{R}{\beta n}.
\end{equation}
In order to prove next proposition, it will be useful to keep in mind the gradient and the Hessian matrix of \eqref{RTPR}. The gradient is
\begin{equation} \label{gradu}
	\nabla u(x) = -\frac{x}{n}.
\end{equation}
In particular, if $x \in \partial B_R$, being $\nu = \frac{x}{R}$ the outer unit normal to the boundary of $B_R$, then
\begin{equation} \label{boundarygrad}
	\nabla u(x) = -\frac{R}{n}\nu, \,\,\,\,\,\,\,\,\,\, \frac{\partial u}{\partial\nu}  = -\frac{R}{n}.
\end{equation}
The Hessian matrix is clearly negative definite and it is given by
\begin{equation}\label{hessu}
 \text{Hess}_u(x)= -\frac{I}{n},
\end{equation}
where I is the identity matrix.

\begin{prop}
Let $V$ be a $C^{2,\alpha}(\mathbb{R}^n,\mathbb{R}^n)$ vector field and $R>0$ a positive real number. The function $u'$ solves the following boundary value problem in the ball with radius $R$
\begin{equation} \label{ELSD}
\begin{cases}
\Delta u' = 0 & \mbox{in}\ B_R\vspace{0.2cm}\\
\ds \frac{\partial u'}{\partial\nu} +\beta u' = \ds \left( \frac{1+\beta R}{n}\right) (V\cdot \nu) &\mbox{on}\ \partial B_R\vspace{0.2cm},
\end{cases}
\end{equation}
where $\nu$ is the outer unit normal to $\partial B_R$. 
\end{prop}
\begin{proof}
If we differentiate  the first equation in \eqref{ELRT} with respect to $t$ and evaluate for $t=0$, we obtain
\begin{equation*}
	L_{A(0)}\dot{\tilde{u}}(0)+ L_{\dot{A}(0)}\tilde{u}(0)+\dot{J}(0)=0.
\end{equation*}
It will be helpful to write explicitly $L_A$ and $L_{\dot{A}}$. By applying Lemma \ref{EXPANSIONS} we have
\begin{equation*}
\begin{split}
	\frac{\partial}{\partial x_j}\bigg(\delta_{ij} \frac{\partial}{\partial x_i}\bigg)&\bigg(u'+V_k\frac{\partial u}{\partial x_k}\bigg)\\
	& + \frac{\partial}{\partial x_j}\bigg(\frac{\partial V_k}{\partial x_k} \delta_{ij}\frac{\partial}{\partial x_i} -\frac{\partial V_j}{\partial x_i} \frac{\partial}{\partial x_i}-\frac{\partial V_i}{\partial x_j}\frac{\partial}{\partial x_i}\bigg)u + \frac{\partial V_k}{\partial x_k} = 0.
	\end{split}
\end{equation*}
Hence
\begin{equation*} 
\begin{split}
	\frac{\partial}{\partial x_j}\bigg( \frac{\partial u'}{\partial x_j}+\frac{\partial}{\partial x_j} & \bigg(V_k\frac{\partial u}{\partial x_k}\bigg)\bigg)\\
	 &+ \frac{\partial}{\partial x_j}\bigg(\frac{\partial V_k}{\partial x_k} \frac{\partial u}{\partial x_j} -\frac{\partial V_j}{\partial x_i} \frac{\partial u}{\partial x_i}-\frac{\partial V_i}{\partial x_j}\frac{\partial u}{\partial x_i}\bigg) + \frac{\partial V_k}{\partial x_k} = 0.
\end{split}
\end{equation*}
Renaming the indexes $k$ 
\begin{equation*}
\begin{split}
	\frac{\partial^2 u'}{\partial x_j^2} &+\frac{\partial^2 V_i}{\partial x_j^2} \frac{\partial u}{\partial x_i}+2\frac{\partial V_i}{\partial x_j}\frac{\partial^2 u}{\partial x_j\partial x_i} + V_i \frac{\partial^3 u}{\partial^2x_j \partial x_i}\\
	&+\frac{\partial^2 V_i}{\partial x_j \partial x_i}\frac{\partial u}{\partial x_j} +\frac{\partial V_i}{\partial x_i} \frac{\partial^2 u}{\partial x_j^2}-\frac{\partial^2 V_j}{\partial x_j\partial x_i}\frac{\partial u}{\partial x_i}  \\
	&-\frac{\partial V_j}{\partial x_i}\frac{\partial^2 u}{\partial x_j \partial x_i}-\frac{\partial^2 V_i}{\partial x_j^2}\frac{\partial u}{\partial x_i} 
	-\frac{\partial V_i}{\partial x_j}\frac{\partial^2 u}{\partial x_j \partial x_i} + \frac{\partial V_i}{\partial x_i} = 0.
\end{split} 
\end{equation*}
Considering that $(\mathrm{Hess}_u(x))_{ij} = \frac{\partial^2 u}{\partial x_j \partial x_i}=0$ whenever $i\neq j$ and the fact that $\frac{\partial^2 u}{\partial x_j^2}=\Delta u= -1$, we have
\begin{equation*}
\frac{\partial u'}{\partial x_j^2} + \frac{\partial^2 V_i}{\partial x_j\partial x_i}\frac{\partial u}{\partial x_j}-\frac{\partial^2 V_j}{\partial x_j\partial x_i} \frac{\partial u}{\partial x_i}  = 0. 
\end{equation*}
Eventually, by Schwarz Theorem, we obtain $\Delta u' = 0$ in $B_R$. \\
If we now differentiate the boundary conditions in \eqref{ELRT} and evaluate them for $t=0$, then
\begin{equation} \label{BOUNDARY}
	\partial_{\nu_{A(0)}}\dot{\tilde{u}}(0) + \partial_{\nu_{\dot{A}(0)}}\tilde{u}(0) + \beta\dot{m}(0)\tilde{u}(0)+ \beta m(0)\dot{\tilde{u}}(0)=0.
\end{equation}
Let us compute every term in the previous equation. Considering the boundary conditions satisfied by $u$, Lemma \eqref{EXPANSIONS} and \eqref{boundarygrad} we get
\begin{equation*}
\partial_{\nu_{A(0)}}\dot{\tilde{u}}(0) =\frac{\partial u'}{\partial \nu}  + \frac{\partial}{\partial \nu}(V\cdot \nabla u) = \frac{\partial u'}{\partial \nu} -\frac{R}{n}\nabla(V\cdot \nu)\cdot \nu,
\end{equation*}
\begin{equation*}
\begin{split}
	\partial_{\nu_{\dot{A}(0)}}\tilde{u}(0) &= \frac{\partial u}{\partial \nu} \, \mathrm{div}\,V - \nu \cdot D_V \nabla u - \nabla u \cdot D_V\nu\\
	& =  -\frac{R}{n} \mathrm{div}\,V +\frac{2R}{n}\nu \cdot D_V \nu,
\end{split}   
\end{equation*}
\begin{equation*}
	 \beta\dot{m}(0)\tilde{u}(0) = \beta u \, \mathrm{div}_{\partial B_R} V = \frac{R}{n}\mathrm{div}V-\frac{R}{n}\nu\cdot D_V \nu,
\end{equation*}
\begin{equation*}
	\beta m(0)\dot{\tilde{u}}(0) = \beta u' + \beta V\cdot \nabla u =\beta u' -\frac{R\beta}{n} V\cdot \nu.
\end{equation*}
Substituting in \eqref{BOUNDARY} we have 
\begin{equation*}
\frac{\partial u'}{\partial \nu} + \beta u' =
 \frac{R}{n}\nabla(V\cdot \nu)\cdot \nu -\frac{R}{n}\nu \cdot D_V\nu + \frac{\beta R}{n} V\cdot \nu.
\end{equation*}
Now
\begin{equation*}
\frac{R}{n}\nabla (V\cdot \nu) =\frac{R}{n} \nu \cdot D_V \nu + \frac{1}{n} V\cdot \nu.
\end{equation*}
Hence
\begin{equation*}
	\frac{\partial u'}{\partial \nu} + \beta u' = \frac{1}{n} V\cdot \nu + \frac{\beta R}{n} V\cdot \nu =  \left( \frac{1+\beta R}{n}\right) (V\cdot \nu).
\end{equation*}
\end{proof}

As a consequence of the previous proposition, we deduce that
\begin{cor} \label{ZMF}
	If  $\,V\in C^{2,\alpha}(\mathbb{R}^n,\mathbb{R}^n)$ is a vector field volume preserving of the first order, then the solution $u'$ to \eqref{ELSD} is a function with zero mean value in $B_R$, that is
	\begin{equation*}
	\int_{B_R} u'\,dx = \int_{\partial B_R} u' \,d\sigma=0.
	\end{equation*}
	\begin{proof}
		Let us integrate the first equation in \eqref{SDB}
		\begin{equation} \label{NDUP}
		0=\int_{B_R} \Delta u' \,dx = \int_{B_R} \mathrm{div}(\nabla u') \,dx = \int_{\partial B_R} \frac{\partial u'}{\partial \nu} \,d\sigma.
		\end{equation}
		Since $V$ is volume preserving of the first order
		\begin{equation} \label{VPFO}
		\int_{B_R} \mathrm{div}\, V \,dx= \int_{\partial B_R} V\cdot \nu \,d\sigma =0.
		\end{equation}
		So integrating the second equation on the boundary, by \eqref{NDUP} and \eqref{VPFO}
		\begin{equation*}
		\int_{\partial B_R} u' \,d\sigma = \frac{1}{\beta}\left( \frac{1+\beta R}{n} \right) \int_{\partial B_R} V\cdot \nu \,d\sigma =0.
		\end{equation*}
		In conclusion, being $u'$ a harmonic function, by the mean value theorem
		\begin{equation*}
		\int_{B_R} u' \,dx = \frac{R}{n} \int_{\partial B_R} u' \,d\sigma = 0.
		\end{equation*}
	\end{proof}
\end{cor}

\section{Main results}
The shape derivative $u'$, solution to problem \eqref{ELSD} will play a central role to prove the desired results.\\
Let us recall the definition of the $L^{\infty}$ and $L^p$ functionals. If $\Omega$ is a bounded and simply connected open subset of $\mathbb{R}^n$ and $u$ is the solution to problem \eqref{TPR}, we will denote the $L^{\infty}$-functional over $\Omega$ by
\begin{equation*}
	M(\Omega)= \| u\|_{L^{\infty}(\Omega)},
\end{equation*}
and, if $p \ge 1$, the $L^p$-functional
\begin{equation*}
	F_p (\Omega)= \|u\|_{L^p(\Omega)}^p.
\end{equation*}
\subsection{Shape derivative of the $L^{\infty}$-norm}
We prove the next result following the proof that can be found in [HLP].
\begin{thm} \label{SDM}
Let $B_R$ be a ball centered at the origin with radius $R>0$. Then for every $C^{2,\alpha}(\mathbb{R}^n, \mathbb{R}^n)$ vector field $V$, the shape derivative of $M$ at $B_R$ in any direction $V$ exists and it is given by
\begin{equation*}
M'(B_R, V) = u'(0),
\end{equation*}
where $0$ is the maximum point of \eqref{RTPR} and $u'$ is the solution to \eqref{ELSD}.
\end{thm}
\begin{proof} 
Let us perturb the ball $B_R$ as we have seen in the previous section
\begin{equation*}
	B_{R,t} = \{ y = x+tV(x): x \in B_R,\,\, t \,\, \text{small enough}\},
\end{equation*}
and consider problem \eqref{ELR}, with $B_{R,t}$ in place of $\Omega_t$.\\
Being $x=0$ the unique maximum point of \eqref{RTPR}, then $\nabla u(0,0)=0$. On the other hand, by the strict concavity of the torsion function $u$ on $B_R$, the matrix
\begin{equation*}
D_y\nabla u (0,0) = \mathrm{Hess}_u(0)
\end{equation*}
is invertible, since $\mathrm{Hess}_u$ is negative definite. Hence by the implicit function theorem, in a neighbourhood of the origin and for $t$ small enough, there exists a unique $y_t$ such that $\nabla u(y_t,t)=0$. Moreover the function $t\longrightarrow y_t$ is differentiable and $y_t$ must be a maximum, so $M(B_{R,t})=u(y_t,t)$.\\
We want to prove that
\begin{equation*}
\lim_{t\to 0 } \frac{M(B_{R,t})-M(B_R)}{t} = u'(0),
\end{equation*}
where
\begin{equation*}
\begin{split}
\frac{M(B_{R,t})-M(B_R)}{t} &= \frac{u(y_t,t)-u(0,0)}{t} \\
&= \frac{u(y_t,t)-u(0,t)}{t}+\frac{u(0,t)-u(0,0)}{t}.
\end{split}
\end{equation*}
By the differentiability of the map $t\longrightarrow u(\cdot,t)$ and the the fact that $\nabla u(0,0)=0$, we have
\begin{equation*}
\lim_{t\to 0 } \frac{u(0,t)-u(0,0)}{t} = \frac{d}{dt}\left[ u(0,t)\right]_{t=0} = u'(0)+ V(0)\cdot \nabla u(0,0) = u'(0).
\end{equation*}
\\
Furthermore, by the differentiability of $t\longrightarrow y_t$, by Lagrange theorem on the segment $[0, y_t]$, the mean value property of $\nabla u(\cdot,t)$ and the regularity of $u(\cdot,t)$, we get
\begin{equation*}
\begin{split}
\frac{u(y_t,t)-u(0,t)}{t} &= \nabla u (\xi_t,t) \frac{y_t}{t} \\& 
 =\left(\frac{1}{|B_r(\xi_t)|} \int_{B_r(\xi_t)} \nabla u(y,t)\,dy \right) \frac{y_t}{t},
\end{split}
\end{equation*}
with $\xi_t$ a suitable point in $[0,y_t]$. Hence
\begin{equation*}
	\lim_{t\to 0 } \frac{u(y_t,t)-u(0,t)}{t} = \nabla u(0,0) \left[ \frac{d y_t}{dt}\right]_{t=0}  = 0.
\end{equation*}
This conclude the proof.
\end{proof}

\begin{cor} \label{SBI}
The ball is a critical shape for the functional $M$, for every $V \in C^{2,\alpha}(\mathbb{R}^n,\mathbb{R}^n)$ which is volume preserving of the first order, i.e.
\begin{equation*}
M'(B_R, V) = u'(0)=0.
\end{equation*}
\end{cor}
\begin{proof} As a consequence of Lemma \ref{ZMF} and Theorem \ref{SDM}, applying the mean value theorem, we have that
\begin{equation*}
M'(B_R,V)=u'(0)= \frac{1}{n\omega_n R^{n-1}}\int_{\partial B_R} u'(x)\,d\sigma =\frac{1}{\omega_n R^n}\int_{B_R} u'\,dx= 0.
\end{equation*}
\end{proof}

\subsection{Shape derivative of the $L^p$-norm}
Next theorem will be a straightforward computation of the shape derivative of the functional $F_p(\Omega)$.
\begin{thm}\label{SDFP} For every $C^{2,\alpha}(\mathbb{R}^n, \mathbb{R}^n)$ vector field $V$, the shape derivative of $F_p$ at $\Omega$ in any direction $V$ exists and it is given by
\begin{equation*}
F'_p(\Omega, V) = p\int_{\Omega} u^{p-1}u'\,dx +\int_{\partial \Omega} u^p (V\cdot \nu)\,d\sigma ,
\end{equation*}
where $u'$ is the shape derivative of $u$, solution to problem \eqref{TPR}, and $\nu$ is the outer unit normal to $\partial\Omega$.

\end{thm}
\begin{proof} Let $u$ be the solution to the perturbed problem
\begin{equation*}
\begin{cases}
-\Delta_y u(y,t) = 1 & \mbox{in}\ \Omega_t\vspace{0.2cm}\\
\dfrac{\de u}{\de \nu_t}(y,t)+\beta u(y,t) =0&\mbox{on}\ \partial \Omega_t\vspace{0.2cm},
\end{cases}
\end{equation*}
where $\Omega_t$ is the perturbed domain defined in \eqref{PERT}. Then
\begin{equation*}
F_p(\Omega_t) = \int_{\Omega_t} u^p(y,t)\,dy = \int_{\Omega} \tilde{u}^p(t) J(t)\,dx,
\end{equation*}
with $\tilde{u}(t) = \tilde{u}(x+tV(x),t)$ and $J(t)$ the Jacobian determinant as in \eqref{JD}.

Then it is possible to differentiate under the sign of integral and
\begin{equation*}
\begin{split}
\frac{d}{dt} F_p(\Omega_t)&= \frac{d}{dt} \int_{\Omega} \tilde{u}^p(t) J(t)\,dx = \int_{\Omega}\frac{d}{dt}[ \tilde{u}^p(t) J(t)]\,dx\\
&= p \int_{\Omega} \tilde{u}^{p-1}(t)\dot{\tilde{u}}(t) J(t)\,dx + \int_{\Omega} \tilde{u}^p(t)\dot{J}(t) \,dx.
\end{split}
\end{equation*}
Evaluating this derivative for $t=0$
\begin{equation*}
\left[ \frac{d}{dt} F_p(\Omega_t)\right]_{t=0} = p \int_{\Omega} \tilde{u}^{p-1}(0)\dot{\tilde{u}}(0) J(0)\,dx \\
+ \int_{\Omega} \tilde{u}^p(0)\dot{J}(0) \,dx.
\end{equation*}
Applying Lemma \eqref{EXPANSIONS}, we get
\begin{equation*}
\begin{split}
F'_p(\Omega,V) &= p\int_{\Omega} u^{p-1}u'\,dx +p\int_{\Omega} u^{p-1} V\cdot \nabla u \,dx + \int_{\Omega} u^p \mathrm{div}\,V \,dx \\
&=p\int_{\Omega} u^{p-1}u'\,dx +p\int_{\Omega} u^{p-1} V\cdot \nabla u \,dx+ \int_{\partial\Omega} u^p (V\cdot \nu)\,d\sigma \\
&- p\int_{\Omega} u^{p-1} V\cdot \nabla u \,dx = p\int_{\Omega} u^{p-1}u'\,dx + \int_{\partial\Omega} u^p (V\cdot \nu)\,d\sigma.
\end{split}
\end{equation*}
\end{proof}

When $\Omega = B_R$, we can use the symmetry properties of \eqref{RTPR} and the property of $u'$ to be a zero mean function, to prove that
\begin{cor}
The ball $B_R$ centered at the origin with radius $R>0$ is a critical shape for the functional $F_p$, for every $1 \le p < +\infty$ and every vector field $V \in C^{2,\alpha}(\mathbb{R}^n,\mathbb{R}^n)$ which is volume preserving of the first order, i.e.
\begin{equation*}
F'_p(B_R, V) =0.
\end{equation*}
\end{cor}
\begin{proof} 
By previous theorem we know that
\begin{equation*}
F'_p(B_R,V) = p\int_{B_R} u^{p-1} u'\,dx + \int_{\partial B_R} u^p (V\cdot \nu)\,d\sigma.
\end{equation*}
Being $u$ constant on the boundary and $V$ a vector field volume preserving of the first order
\begin{equation*}
\int_{\partial B_R} u^p (V\cdot \nu)\,d\sigma = u_{\min}^p \int_{\partial B_R}(V\cdot \nu)\,d\sigma = u_{\min}^p\int_{B_R} \mathrm{div}\,V =0,
\end{equation*}
where $u_{\min} = u(R) = \frac{R}{\beta n}$.
By corollary \ref{SBI}, we know that $u'(0)=0$ and so by the mean value theorem, we have that
\begin{equation*}
	\int_{\partial B_r} u'(x)\,d\mathcal{H}^{n-1} = u'(0)n\omega_n r^{n-1} = 0,
\end{equation*}
for every $r \in [0;R]$. Eventually, applying the Coarea Formula
\begin{equation*}
\begin{split}
	&p\int_{B_R} u^{p-1}u'\,dx = p\int_{B_R} \left(\frac{R}{\beta n} - \frac{1}{2n}(R^2-|x|^2)\right)^{p-1} u' \,dx \\
	=& \, p\int_0^R \left(\frac{R}{\beta n} - \frac{1}{2n}(R^2-r^2)\right)^{p-1} \int_{\partial B_r} u'(x)\,d\sigma\,dr = 0.
\end{split}
\end{equation*}
Hence
\begin{equation*}
F'_p(B_R,V)= 0.
\end{equation*}
\end{proof}

\small{

}
\end{document}